\documentclass{article}
\usepackage[latin1]{inputenc}
\usepackage[english]{babel}
\usepackage[T1]{fontenc}
\usepackage{amssymb}
\usepackage{amsmath,amsfonts,amscd,authblk}
\usepackage{mathtools}
\usepackage{graphicx}
\usepackage{indentfirst}
\usepackage{dsfont}
\usepackage{color}
\usepackage{enumitem}
\usepackage{float}
\usepackage{pict2e}
\usepackage{amsthm}
\usepackage{multirow}
\title{Test comparison for Sobol indices over nested sets of variables}

\author[1]{Thierry Klein} 
\author[2]{Nicolas Peteilh}
\author[3]{Paul Rochet}

\affil[1]{Institut de Math\'ematiques de Toulouse; UMR5219. Universit\'e de Toulouse}
\affil[1,2,3]{ENAC - Ecole Nationale de l'Aviation Civile , Universit\'e de Toulouse, France.}

\newtheorem{theorem}{Theorem}[section]
\newtheorem{lemma}[theorem]{Lemma}
\newtheorem{prop}[theorem]{Proposition}

\newcommand{\var}{\operatorname{var}}

\begin{document}

\maketitle

\begin{abstract}
Sensitivity indices are commonly used to quantify the relative influence of any specific group of input variables on the output of a computer code. One crucial question is then to decide whether a given set of variables has a significant impact on the output. Sobol indices are often used to measure this impact but their estimation can be difficult as they usually require a particular design of experiment. In this work, we take advantage of the monotonicity of Sobol indices with respect to set inclusion to test the influence of some of the input variables. The method does not rely on a direct estimation of the Sobol indices and can be performed under classical iid sampling designs.

\end{abstract}

\textbf{Keywords}: Global sensitivity indices, Sobol indices, significance test

\medskip

\textbf{AMS subject classification}  60F05,62G05, 62G20, 62E20, 62F03, 62F05.

\section{Introduction}

The use of complex computer models for the analysis of applications from  sciences, engineering and other fields is by now routine. For instance, in the area of marine submersion, complex computer codes have been  developed  to simulate submersion events (see e.g.~\cite{betancourt:hal-01998724,idier:hal-02458084}) while sensitivity analysis and meta-modelling  are intensively used to optimize the airplanes designs \cite{peteilh:hal-02866381}. Meta-models usually depend on many input variables and are computationally expensive. Thus, it is  crucial to understand which of the input parameters have an  influence  on  the  output. One classical approach to deal with this kind of  problem is to consider the   inputs as random elements, a point of view  generally called (global) sensitivity analysis. We refer to \cite{rocquigny2008uncertainty,saltelli-sensitivity,sobol1993} for an overview of the practical aspects. \\
 
Sobol indices, based on the Hoeffding decomposition \cite{Hoeffding48} of the output's variance, are one of the most used tools to perform global sensitivity analysis. They were first introduced in \cite{pearson1915partial} and later revisited in \cite{sobol2001global}. In the general framework, a square integrable output variable $Y$ is assumed to obey a non-parametric relation of the form

\begin{equation}\label{eq:mod}
Y = f(Z_1,\dots,Z_q) 
\end{equation}
where the $Z_j$'s are  input variables. In practice, an analytical expression for $f$ is usually not available and the only access we have to $f$ is through experimentation or computer code. 
One of the main tasks the practitioner has to deal with is to decide whether a group of variables has any influence on the output $Y$. An effective way to measure the influence of a subset $u \subset \{1,\dots,p\}$ of input variables is to consider the Sobol index of $Y$ with respect to $X_j, j \in u$, defined by
$$ S^{(u)}: = \frac{\var\left(\mathbb E[Y|Z_j,j\in u]\right)}{\var(Y)}. $$
It is easy to see that the Sobol index $S^{(u)}$ is zero if, and only if, the conditional expectation $\mathbb E[Y|Z_j,j\in u]$ is constant, in which case one can naturally consider that the inputs $Z_j, j\in u$ have no direct influence on $Y$ (although they may have an impact on $Y$ through interactions with other variables). On the other hand, the extra information carried by additional inputs can be quantified by the resulting increase in the Sobol index $S^{(v)}$ for $v \supset u$. Because the equality of Sobol indices for nested sets of inputs in equivalent to the (almost sure) equality of the conditional expectations: 
$$ \forall u,v \ , \  u \subset v \ : \ S^{(u)} = S^{(v)} \iff \mathbb E [ Y | Z_j, j \in u] \ \overset{a.s.} = \ \mathbb E[Y | Z_j, j \in v], $$  
a natural notion of non-parametric significance can be established by comparing Sobol indices over nested sets of input variables. \\

Many different estimation procedures of the Sobol indices have been proposed in the literature. Some are based on Monte-Carlo or quasi Monte-Carlo designs of experiment, see \cite{Kucherenko2017different,Owen13}. More recently, a method based on nested Monte-Carlo \cite{GODA201763} has been developed. Other estimation procedures are based on different designs of experiment using for example polynomial chaos expansions \cite{Sudret2008global}. An efficient estimation of the Sobol indices can be performed through the so-called ``Pick-Freeze'' method, whose theoretical properties (consistency, central limit theorem, concentration inequalities and Berry-Esseen bounds) have been studied in \cite{pickfreeze, janon2012asymptotic}. In particular, the joint central limit theorem enables to build asymptotic comparison tests on Sobol indices. However, the Pick-Freeze method requires a specific design of experiment which makes it inapplicable in the classical iid framework and computationally expensive (for instance, the $p$ order one Sobol indices estimators need $n(p+1)$ computations of the function $f$). This drawback was recently partially solved in \cite{ggkl21}, where the order one Sobol indices are estimated from rank statistics in the classical iid sample scheme. Nevertheless, the absence of a joint CLT in this case makes it impossible to test hypotheses involving more than one Sobol index at a time.  \\

In this work, we present an alternative way to build non-parametric significance tests, used for the detection of non-influent variables. A main motivation of the proposed procedure is to perform non-parametric variable selection, in order for instance to reduce the cost of a computational code or simplify a meta-model. The originality  of our approach stems from  a reformulation of the null hypothesis in terms of the empirical process, thus bypassing the difficulty of  having to estimate the  Sobol indices. This allows to perform multiple significance tests using a single sample, thus potentially reducing the computational cost compared to alternative methods that rely on specific sampling designs. The framework and theoretical setting are presented in Section \ref{sec:framework}, while the construction of the test is detailed in Section \ref{sec:test}. In Section \ref{sec:appli}, we show a numerical study comparing the performances of the test procedure to the classical one introduced in \cite{pickfreeze, janon2012asymptotic}, and describe a step-by-step approach for non-parametric variable selection applied to aeronautical data.

\section{Theoretical framework}\label{sec:framework}
We consider the model 
\begin{equation}\label{eq:mod1bis}
Y \overset{a.s.} =f(Z_1,\ldots, Z_q)
\end{equation}
where $f$ is an unknown function, $Y$ is a square integrable real random variable and $Z_1,...,Z_q$ are independent real valued inputs. For any subsets $u,v$ of $\{1,\ldots,q\}$ such that $\emptyset \subseteq u \subset v$ we are interested in testing
$$ H_0: \mathbb E[Y | Z_j, j \in u] \ \overset{a.s.}{=} \ \mathbb E[Y | Z_j, j \in v].  $$
This setting can be viewed as a non-parametric significance test for the variable $(Z_j)_{j \in v \setminus u}$ in presence of $(Z_j)_{j \in u}$, where the influence of an input is measured by its impact on the conditional expectation function. As an important particular case, the importance of the inputs $Z_j, j \in v $ on $Y$ can be investigated by setting $u = \emptyset$. Without loss of generality, we may rewrite the model as 
\begin{equation}\label{eq:mod2}
Y \overset{a.s.} = f(X,W)
\end{equation}
where $X := (Z_j)_{j \in v} $ represents the inputs of interest to explain $Y$ and $W := (Z_j)_{j \notin v}$ contains all the variables that are irrelevant to the analysis. Alternatively, $W$ may contain hidden random inputs encountered in the context of stochastic codes \cite{mazo2021trade, FKL21}. For our purposes however, the precise nature of $W$ is not important as both situations are dealt with in the same way. \\

Since $Y$ is square-integrable, the conditional expectation $\mathbb E[Y | X]$ can be defined as an orthogonal projection of $Y$ in $\mathbb L^2$ onto the linear space of square integrable mesurable functions of $X$. In particular, the Hoeffding decomposition 
$$ \var(Y) = \var \big( \mathbb E [Y | X ] \big) +   \var \big( Y -  \mathbb E [Y | X ] \big)   $$
follows from the Pythagorean theorem. The Sobol index of $Y$ associated to $X$ is defined by
$$ S = \frac{\var \big( \mathbb E [Y | X] \big) }{\var(Y)} \in [0;1]. $$
For any  subset $u$  of $\{1,...,p\}$ and $x = (x_1,...,x_p) \in \mathbb R^p$, we denote $x^{(u)} = (x_j)_{j \in u}$ with the convention $x^{(u)} = 0$ if $u = \emptyset$, and by $S^{(u)}$ the Sobol index associated to $X^{(u)}$:
\begin{equation}\label{eq:sob}
S^{(u)} = \frac{\var \big( \mathbb E [Y | X^{(u)}] \big) }{\var(Y)} \in [0,1]. 
\end{equation}
The influence of the inputs $X_j, j \notin u$ on the conditional expectation function can be assessed by considering the null hypothesis
\[
H_0\ :   S^{(u)}=  S \qquad \textrm{against}  \qquad H_1:\ S^{(u)} < S.
\]

A natural test for $H_0$ exists whenever one can construct estimators of the Sobol indices with known (or estimable) joint limit distribution. However, typical estimation methods such as Pick-Freeze usually requires a specific design of experiment. The test procedure proposed in this paper does not rely on a direct estimation of the Sobol indices and applies in the typical iid sampling design. The method makes use of an equivalent formulation of $H_0$ described in Lemma \ref{lem:statequiv} below.\\

For two vectors $a = (a_1,...,a_k), b = (b_1,...,b_k)$, $a \wedge b := (a_1 \wedge b_1, ..., a_k \wedge b_k)$ denotes the component-wise minimum, while the inequality $a \leq b$ is meant as $a \wedge b = a$. The indicator function is denoted by $\mathds 1 \{ . \}$.

\begin{lemma}\label{lem:statequiv} If $X_1,...,X_p,W$ are independent, then for all $ u \subseteq \{1,...,p\}$, the following assertions are equivalent:
\begin{itemize}
    \item [i)] $ S^{(u)}= S$.
    \item [ii)] $\mathbb E [ Y|X^{(u)} ] = \mathbb E [Y | X]$.
    \item [iii)] For all $x \in \mathbb R^p$ such that $\mathbb P(X \leq x) > 0$, $\mathbb E [Y | X^{(u)} \leq x^{(u)}] = \mathbb E[Y|X \leq x]$.
 \item [iv)] For all $x \in \mathbb R^p$, $\mathbb E [Y \mathds 1 \{ X \leq x \} ] = \mathbb E [Y \mathds 1 \{ X^{(u)} \leq x^{(u)} \} ] \, \mathbb P(X^{(\overline u)} \leq x^{(\overline u)})$.
\end{itemize}
\end{lemma}

\begin{proof} Let $\phi(X) = \mathbb E [ Y | X ]$ and remark that $\mathbb E[Y | X^{(u)}] = \mathbb E [\phi(X) | X^{(u)} ]$.  The equivalence between $i)$ and $ii)$ follows from the well known decomposition
$$ \var \big( \phi(X) \big) = \var \big( \mathbb E \big[ \phi(X) | X^{(u)} \big] \big) + \mathbb E \big[ \var \big(\phi(X) | X^{(u)} \big) \big], $$
where the non-negative term $\mathbb E \big[ \var \big(\phi(X) | X^{(u)} \big) \big]$ is zero if, and only if, $\mathbb E[ \phi(X) | X^{(u)}] = \phi(X)$. By definition of the conditional expectation
$$ \mathbb E [Y | X] = \mathbb E[Y | X^{(u)}] \iff \forall x \in \mathbb R^p \ , \ \mathbb E \big[ Y \mathds 1 \{ X \leq x \} \big] = \mathbb E \big[ \mathbb E[Y | X^{(u)}]  \mathds 1 \{ X \leq x \} \big]. $$
The independence of the $X_j$'s and the fact that $\mathds 1 \{ X \leq x \} = \mathds 1 \{ X^{(u)} \leq x^{(u)} \} \mathds 1 \{ X^{(\overline u)} \leq x^{(\overline u)}  \}$ give
$$  \mathbb E \big[ \mathbb E(Y | X^{(u)})  \mathds 1 \{ X \leq x \} \big] = \mathbb E \big[ Y \mathds 1 \{ X^{(u)} \leq x^{(u)} \} \big] \, \mathbb P \big(X^{(\overline u)} \leq x^{(\overline u)} \big) $$
which shows $ii) \iff iv)$. The equivalence $ iii) \iff iv)$ follows by dividing both sides of the equality $iv)$ by $\mathbb P(X \leq x)$.
\end{proof}

Assume we observe an iid sample $(Y_1,X_1),...,(Y_n,X_n)$ drawn from the same distribution as $(Y,X)$. 
For all $ k \in \{0,1,2 \}$ and $u \subseteq \{1,...,p\}$, let $m_k^{(u)}: x \mapsto \mathbb E[ Y^k \mathds 1 \{ X^{(u)} \leq x^{(u)}\}] $ and denote by $\widehat m_k^{(u)}(.)$ its empirical counterpart:
$$ \widehat m_k^{(u)}(x) =  \frac 1 n \sum_{i=1}^n Y_i^k \mathds 1 \{ X_i^{(u)} \leq x^{(u)} \} \ , \ x \in \mathbb R^p.  $$
For ease of notation, we shall simply write $m_k$ and $\widehat m_k$ for the case $u = \{ 1,...,p\}$. By Lemma \ref{lem:statequiv}, we know that the null hypothesis $ H_0 : S^{(u)} = S  $ can be stated as $\xi := m_1 - m_1^{(u)} m_0^{(\overline u)}$ being identically zero. In this logic, we use the empirical version $\widehat \xi$ to build a test statistics for $H_0$. \\

\noindent For the next proposition, we denote by $\eta = (m_1, m_1^{(u)}, m_0^{(\overline u)})^\top$ and $\widehat \eta :=(\widehat m_1, \widehat m_1^{(u)}, \widehat m_0^{(\overline u)})^\top$ its empirical counterpart. Moreover, let us write $x^{(u)} \oplus {x'}^{(\overline u)}$ for the vector of $\mathbb R^p$ with components $x_i$ if $i \in u$ and $x'_i$ if $i \notin u$.

\begin{prop}\label{prop:as}
The normalized process $ \sqrt n ( \widehat \eta - \eta )  $ converges in finite-dimensional distribution towards a $3$-dimensional centered Gaussian field indexed by $\mathbb R^p$ with auto-covariance function
$$ \Omega(x, x') := \left[ \begin{array}{ccc} 
m_2(x \wedge x') & m_2 \big( (x \wedge {x'})^{(u)} \oplus x^{(\overline u)} \big) & m_1 \big( x^{(u)} \oplus (x \wedge {x'})^{(\overline u)} \big)  \\ 
m_2 \big( (x \wedge {x'})^{(u)} \oplus {x'}^{(\overline u)} \big) & m_2^{(u)}(x \wedge x') & m_1 \big(x^{(u)} \oplus {x'}^{(\overline u)} \big) \\ 
m_1\big( {x'}^{(u)} \oplus (x \wedge {x'})^{(\overline u)} \big) & m_1 \big( {x'}^{(u)} \oplus x^{(\overline u)} \big) & m_0^{(\overline u)}(x \wedge x') \end{array} \right] - \eta(x) \eta(x')^\top $$
for all $ x,x' \in \mathbb R^p $.
\end{prop}

The proof is a straightforward consequence of the central-limit theorem. The convergence in distribution can be shown without any particular obstacle using Vapnik-Chervonenkis' theory although this result is not needed for the theoretical validity of the test.

\section{The test procedure}\label{sec:test}

Consider the process $\widehat \xi = \widehat m_1 - \widehat m_1^{(u)} \widehat m_0^{(\overline u)}$, whose asymptotic distribution can be derived from the delta method applied to the smooth function $\phi:(s,t,u) \mapsto s - tu$ from $\mathbb R^3$ to $\mathbb R$, using Proposition \ref{prop:as}. Given a fixed collection $\mathbf x = (x_1, ..., x_K) $ of points in $\mathbb R^p$ chosen independently from the sample, the random vector $\widehat \xi(\mathbf x) = \big( \widehat \xi(x_1), ..., \widehat \xi(x_K) \big)^\top$ is asymptotically Gaussian
$$ \sqrt n \big(\widehat \xi(\mathbf x) - \xi(\mathbf x) \big) \xrightarrow[n \to \infty]{d} \mathcal N \big( 0, \Sigma (\mathbf x) \big) $$
with covariance matrix
$$ \Sigma (\mathbf x) := \Big( \nabla \phi(x_k)^\top \Omega(x_k,x_{k'}) \nabla \phi(x_{k'}) \Big)_{k,k' =1,...,K} .$$
The empirical version $\widehat \Sigma(\mathbf x)$ obtained by replacing the functions $m_k^{(u)}$ by their empirical counterparts $\widehat m_k^{(u)}$, is clearly a consistent estimator of $\Sigma(\mathbf x)$ in virtue of the law of large numbers. Since the hypothesis $H_0: S^{(u)} = S$ can be stated equivalently as $H_0: \xi = 0$, a test can be performed by comparing the observed value of $\Vert \widehat \xi(\mathbf x) \Vert$ (for a well chosen norm $\Vert . \Vert $ on $\mathbb R^K$) to the appropriate quantile of the asymptotic distribution under $H_0$. Two natural approaches are then possible:
\begin{enumerate}
\item If $\Vert . \Vert$ is the natural Euclidean norm on $\mathbb R^K$, then under $H_0$,  
\begin{equation}\label{eq:stat_test1} T : = n \Vert \widehat \xi(\mathbf x) \Vert^2 = n \sum_{k=1}^K \widehat \xi(x_k)^2  
\end{equation}
converges in distribution towards a weighted $\chi^2$ distribution with weights given by the eigenvalues $\lambda_1,...,\lambda_K \geq 0$ of $\Sigma(\mathbf x)$. In other words, $n \Vert \widehat \xi(\mathbf x) \Vert^2$ has the same distribution asymptotically (under $H_0$) as $\epsilon^\top \Sigma(\mathbf x) \epsilon$ where $\epsilon$ is a standard Gaussian vector in $\mathbb R^K$. This distribution can be approximated by Monte-Carlo using the empirical version $\widehat \Sigma(\mathbf x)$ instead of the unknown $\Sigma(\mathbf x)$, in order to determine the threshold $\tau_\alpha$ over which the hypothesis is rejected, at any given significance level $\alpha \in (0,1)$. The Monte-Carlo part can be time consuming as a large number of replications may be needed to approximate the asymptotic distribution and corresponding quantile sufficiently well. 

\item A different approach consists in normalizing the vector $\widehat \xi(\mathbf x)$ in order to achieve a true (non-weighted) $\chi^2$ asymptotic distribution under $H_0$. If $\Sigma(\mathbf x)$ is invertible, with inverse $\Gamma( \mathbf x)$, a test statistics
\begin{equation}\label{eq:stat_test2}  T := n \widehat \xi(\mathbf x)^\top \widehat \Gamma(\mathbf x) \widehat \xi(\mathbf x) 
\end{equation}
for $\widehat \Gamma(\mathbf x)$ a consistent estimator of $\Gamma(\mathbf x)$, has the asymptotic distribution $\chi^2(K)$ under $H_0$, as $n \to \infty$. In practice, the naive estimator
$$ \widehat \Gamma(\mathbf x) = \widehat \Sigma(\mathbf x)^{-1} $$
is rarely a good choice, especially if $\widehat \Sigma(\mathbf x)$ is close to singular. In this case, a regularized version of the inverse leads to a better approximation of the asymptotic distribution. Typically, $\widehat \Gamma(\mathbf x)$ can be obtained by truncated singular value decomposition where the eigenvalues of $\widehat \Sigma(\mathbf x)$ below a certain threshold $t$ are ignored (see for instance \cite{engl1996regularization} for further details on inverse matrix regularization). The observed value of the test statistics is then compared to the quantile of the $\chi^2$ distribution with $r = \operatorname{rank}(\widehat \Gamma(\mathbf x))$ degrees of freedom. In the numerical study, we use the regularization threshold $t = 0.1 n^{-1/3} 
\rho ( \widehat \Sigma(\mathbf x)) 
\lambda_1$ where $\lambda_1$ is the largest eigenvalue of $\widehat \Sigma(\mathbf x)$, which ensures in particular that $r \geq 1$. Further details are discussed in Section \ref{sec:appli}.
\end{enumerate}

For both these approaches, the number $K$ of points over which the empirical process $\widehat \xi$ is evaluated is only constrained by the computation time. A larger experimental design $\mathbf x$ may improve the power of the test with no negative impact on the significance level, as we discuss in Section \ref{sec:appli}.\\

In practice, the $x_k$'s may be drawn uniformly on the domain of $X$ if it is bounded, or from an arbitrary distribution $\mu$ on $\mathbb R^p$. In this case, the normalized test statistics can be viewed as a Monte-Carlo approximation of the integral $\int n \widehat \xi^2 d \mu$. Although possible in practice, we do not recommend using the available sample $(X_1,...,X_n)$ as the design due to the poor resulting performance of the test. If the distribution of the $X_i$'s is known to the practitioner, we may use the same distribution to draw the $x_k$'s. Under the alternative $H_1$, the power of the test highly depends on the design $\mathbf x$ (or the underlying distribution $\mu$) which should ideally favor regions of the space for which $\xi$ is far from zero, enabling the test statistics to grow more rapidly to infinity.

\section{Numerical application}\label{sec:appli}

\noindent Let $(Y_1, X_1), ...., (Y_n,X_n)$ be an iid sample on $\mathbb R \times \mathbb R^3$ obeying the relation
$$ Y_i = f(X_i) \ , \ i=1,...,n  $$
where
$$ f(x) =  (2 + x_3 ^4) \sin(x_1) + 7 \sin^2(x_2) \ , \ x = (x_1,x_2,x_3) \in \mathbb R^3. $$
The $X_i $'s are assumed independent with the uniform distribution on $[- \pi, \pi ]^3$. This function is commonly used in sensitivity analysis as a test case and is classically referred to as the Ishigami function. \\

From the two possible approaches discussed in the previous section, we choose the second one due to its faster computation time. Thus, the test statistics is given by
$$ T = n \widehat \xi(\mathbf x)^\top \widehat \Gamma(\mathbf x) \widehat \xi(\mathbf x)   $$
where $\widehat \Gamma(\mathbf x)$ is a regularized inverse of the empirical estimator $\widehat \Sigma(\mathbf x)$, whose precise construction will be detailed below. \\

The experimental design $\mathbf x = (x_1,...,x_K)$ is drawn from the same distribution as the original sample, namely a uniform distribution on $[-\pi, \pi]^3$. We draw $K=10$ points to build the test. This seemingly small value ended up providing satisfactory results while keeping the computation time reasonable. The power of the test can be slightly improved by taking a larger experimental design $\mathbf x$ but the additional time is too much amplified in our framework where numerous replications were made to evaluate the performances of the test. 
For an actual application of the method where the algorithm is run only once, the computation takes at most a few minutes and the size of $\mathbf x$ is not much of a limiting factor. \\

The matrix $\widehat \Gamma(\mathbf x)$ is obtained from a truncated singular value decomposition of $\widehat \Sigma(\mathbf x)$. Precisely, let $ \lambda_1 \geq ... \geq \lambda_K$ be the ordered eigenvalues of $\widehat \Sigma(\mathbf x)$ and consider the singular value decomposition 
$$ \widehat \Sigma(\mathbf x) = P \operatorname{Diag}(\lambda_1,...,\lambda_K) P^\top $$
where 
$P$ is orthogonal (i.e.~$P P^\top = P^\top P = I$). We define 
$$ \widehat \Gamma(\mathbf x) = P \operatorname{Diag}\big( g_t(\lambda_1) ,..., g_t(\lambda_K) \big)  P^\top  $$
where $g_t$ is the so-called truncated SVD filter function $g_t(x) = 1/x$ if $x >t$ and $g_t(x)=0$ otherwise. The test statistics $T$ is then compared to the quantile of the $\chi^2$ distribution with $r$ degrees of freedom, where $r = \operatorname{rank}(\widehat \Gamma(\mathbf x))$ is the number of eigenvalues of $\widehat \Sigma(\mathbf x)$ larger than $t$. The hypothesis is rejected if the observed value of $T$ exceeds the $(1-\alpha)$-quantile of the $\chi^2(r)$ distribution. To ensure that $r > 0$, we choose $t$ equal to a vanishing proportion $\tau_n \in (0,1)$ of the spectral radius $\lambda_1$ of $\widehat \Sigma(\mathbf x) $: 
\begin{equation}\label{tau} t = \tau_n\lambda_1. \end{equation}
The rule of thumb $\tau_n = 0.1 n^{-1/3}$ is used in the simulations. \\

The test statistics and resulting p-values are calculated over $N = 10000$ replications of the experiments. Four different hypotheses are considered:
\begin{enumerate}
\item $H_0: S^{(3)} = 0  \iff \mathbb E[Y | X_3] = \mathbb E[Y]$
\item $H_0: S^{(2,3)} = S^{(2)} \iff \mathbb E[Y | X_2, X_3] = \mathbb E[Y | X_2]$
\item $H_0: S^{(1)} = 0 \iff \mathbb E[Y | X_1] = \mathbb E[Y]$ 
\item $H_0: S^{(1,3)} = S^{(1)} \iff \mathbb E[Y | X_1, X_3] = \mathbb E[Y | X_1]$
\end{enumerate}

As discussed previously, these hypotheses boil down to testing the non-parametric significance of some input variables, e.g.~the first one reduces to testing the influence of $X_3$ on $Y$ while the second one corresponds to testing the influence of $X_3$ in presence of $X_2$. The null hypothesis is true in the first two cases where the simulations aim to evaluate the actual significance level as a function of the nominal value $\alpha$  the test is supposed to achieve. For the last two cases, the null hypothesis is false with actual values of the Sobiol indices being and $S^{(1)}  \approx 0.402$ and $S^{(1,3)}  \approx 0.989 $. The simulations thus aim to evaluate the power of the test in these last two cases.\\

The results are compared with the test built from the Pick-Freeze estimators of the Sobol indices presented in \cite{gamboa2016statistical}. For each scenario, the expression in Equation (3) is used, and the p-value for the unilateral test is calculated. To easily differentiate the results of the two methods in what follows, the Pick-Freeze based test will be abbreviated to PF, while the method introduced in this paper will be referred to as the Empirical Process (EP) test. \\

We represent the probability of rejecting the null hypothesis for all $\alpha \in  [0,1]$ to give a global view of the distribution of the p-value, although, only the discrepancies between the actual and nominal values for $\alpha$ smaller than say $0.1$ (the range of values typically used in practice) are relevant to measure the reliability of the test procedure for practical purposes. The results are computed for three sample sizes $n$ which designate the number of calls to the function $f$. We emphasize that a specific sampling design is needed for the Pick-Freeze method, which is not the case for the EP test. In particular, all four hypotheses can be tested from a unique sample by the EP approach while individual samples need to be generated for each hypothesis for the PF test. In this aspect, the EP test provides a clear advantage to reduce the number of calls to $f$ if multiple hypotheses are to be tested.

\begin{figure}[h!]
\centering
Probability of rejecting $H_0: S^{(3)} = 0$
\includegraphics[width = \textwidth]{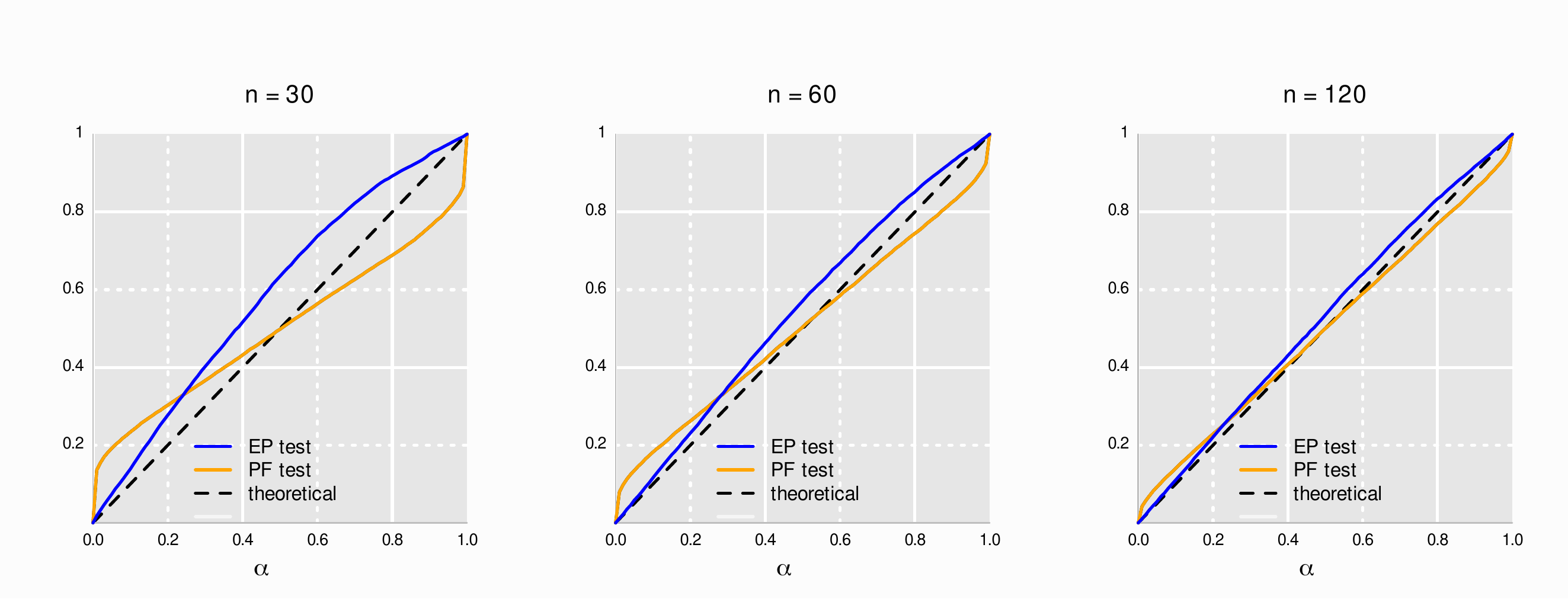}
\caption{\small Estimated probability of rejecting the null hypothesis $H_0: S^{(3)} = 0$ for the Empirical Process (EP) and Pick-Freeze (PF) tests. The empirical cdf of the tests' p-values are calculated on $N = 10000$ iterations and return the (estimated) actual significance level of the test as a function of the nominal level $\alpha$. }
\label{fig:3}
\end{figure}

\begin{figure}[h!]
\centering
Probability of rejecting $H_0: S^{(2,3)} = S^{(2)}$
\includegraphics[width = \textwidth]{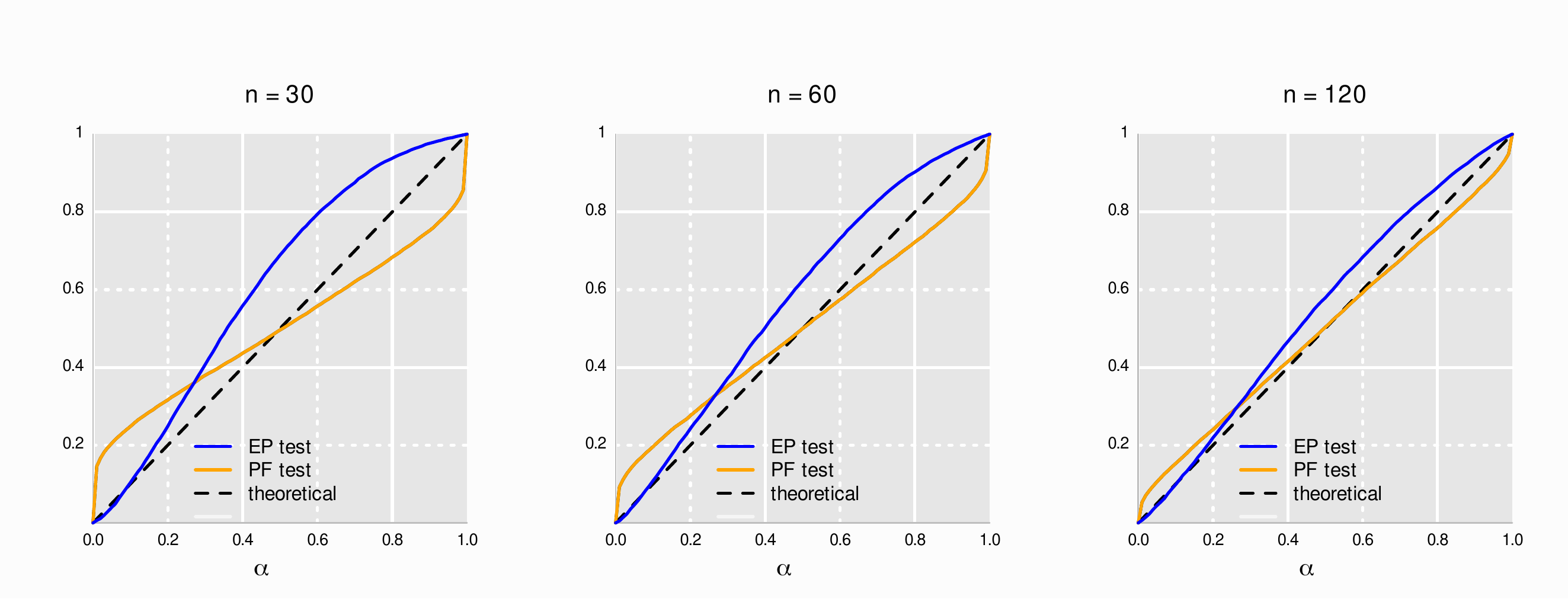}
\caption{\small Estimated probability of rejecting the null hypothesis $H_0: S^{(2,3)} = S^{(2)}$ for the EP and PF tests, as a function of the nominal significance level $\alpha$.}
\label{fig:23}
\end{figure}

As seen in Figures \ref{fig:3}, and \ref{fig:23}, the EP method appears more reliable than the PF approach for the null hypotheses $H_0: S^{(3)} = 0$ and $H_0: S^{(2,3)} = S^{(2)}$, as the (estimated) actual significance level is closer to the nominal value. Here, the rule of thumb with $\tau_n = 0.1 n^{-1/3}$ used for the TSVD regularization of $\widehat \Sigma(\mathbf x)$ seems to yield a well calibrated test for a nominal significance level $\alpha$ below $10 \%$. Unsurprisingly, the discrepancy is more pronounced for small sample sizes. The PF test seems unreliable in these cases as shown by the highly underestimated significance level for small values of $\alpha$. This could be due to a too slow convergence of the Sobol index estimator to a Gaussian distribution, on which the calculations of the critical regions of the PF test are based on. \\

\begin{figure}[h!]
\centering
Probability of rejecting $H_0: S^{(1)} = 0$
\includegraphics[width = \textwidth]{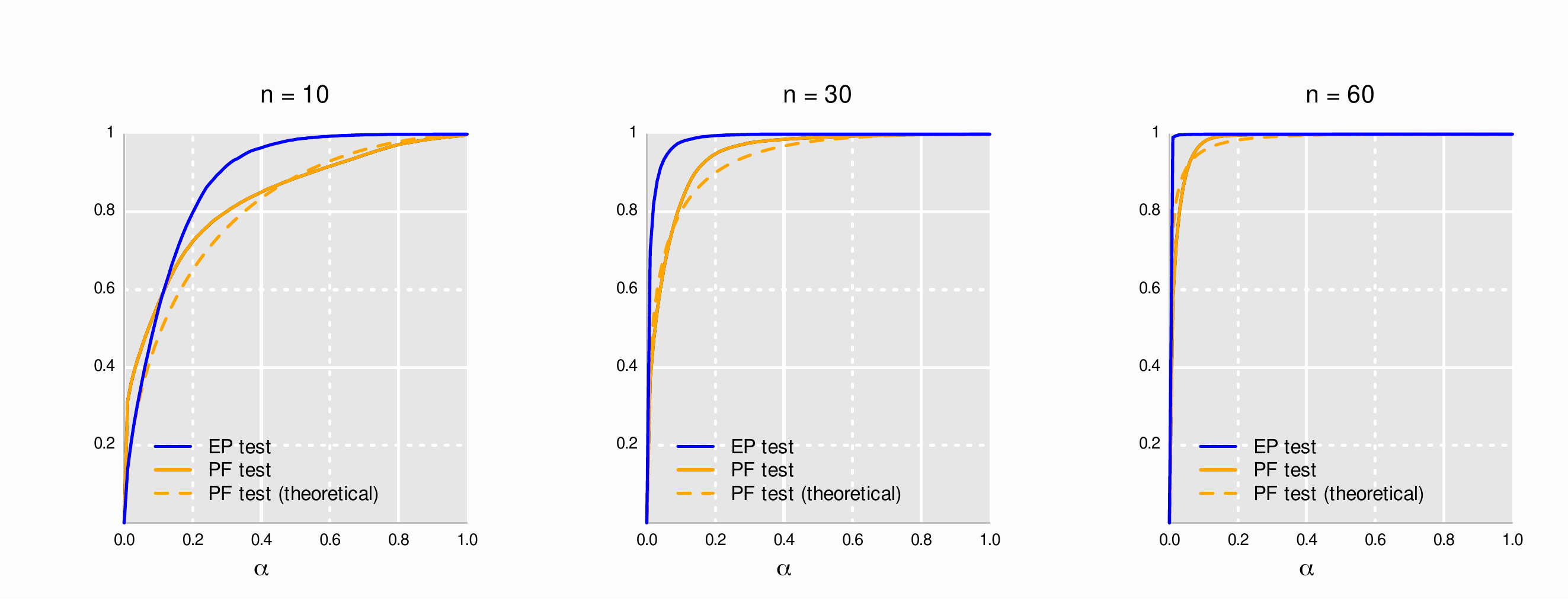}
\caption{\small Estimated probability of rejecting the null hypothesis $H_0: S^{(1)} = 0$ for the EP and PF tests. The orange dashed line gives the asymptotic theoretical power of the PF test obtained under the limit Gaussian distribution of the Pick-Freeze estimator of $S^{(1)}$. In this case where the null hypothesis is not verified ($S^{(1)} \approx 0.402$), the empirical cdf of the tests' p-values returns the estimated power of the test as a function of the nominal significance level $\alpha$.}
\label{fig:1}
\end{figure}

\begin{figure}[h!]
\centering
Probability of rejecting $H_0: S^{(1,3)} = S^{(1)}$
\includegraphics[width = \textwidth]{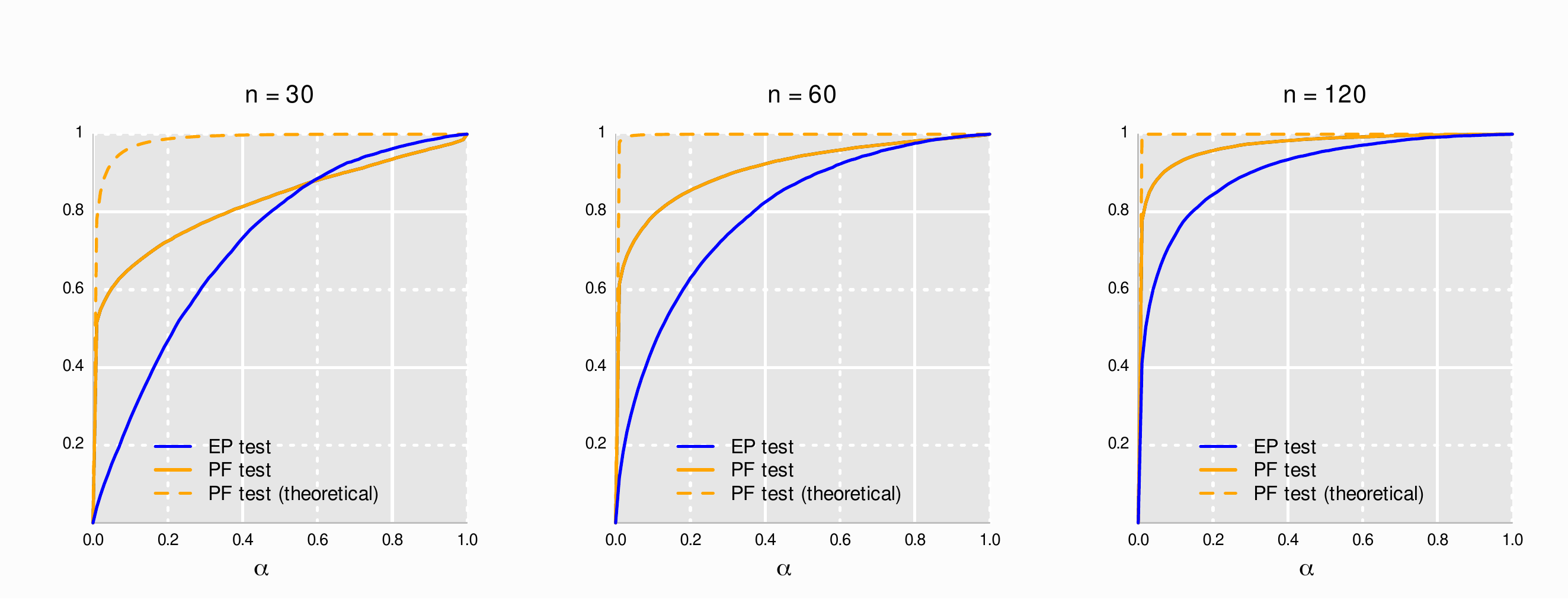}
\caption{\small Power of the EP and PF tests for the null hypothesis $H_0: S^{(1,3)} = S^{(1)}$ as a function of the significance level $\alpha$. The null hypothesis is false in this case where the true values of the Sobol indices actually differ from $S^{(1,3)} - S^{(1)} \approx 0.587$.}
\label{fig:13}
\end{figure}

Figures \ref{fig:1} and \ref{fig:13} display the estimated probability of rightfully rejecting $H_0: S^{(1)} = 0$ and $H_0: S^{(1,3)} = S^{(1)}$ respectively, as a function of the nominal significance level $\alpha$ for the EP and PF tests. The EP test seems to perform better overall for the simple hypothesis $H_0: S^{(1)} = 0$. The power rapidly converges towards $1$ for both tests, which conveys the high (non-parametric) influence of $X_1$ in this situation. The Gaussian approximation used to calibrate the PF test is satisfactory in this case as shown by the theoretical asymptotic power being close to its actual value. On the contrary, the PF test rightfully rejects the null hypothesis $H_0: S^{(1,3)} = S^{(1)}$ more often than the EP test. Despite the relatively high difference $S^{(1,3)} - S^{(1)} \approx 0.587$, the EP test is less powerful than for the previous simple hypothesis $H_0: S^{(1)} = 0$. Nevertheless, while the PF test is more powerful in this case, the convergence to the Gaussian limit appears to be slow as indicated by the high difference between the theoretical asymptotic power and its actual value. 

\begin{figure}[H]
\centering
Probability of rejecting $H_0$
\includegraphics[width = 0.85\textwidth]{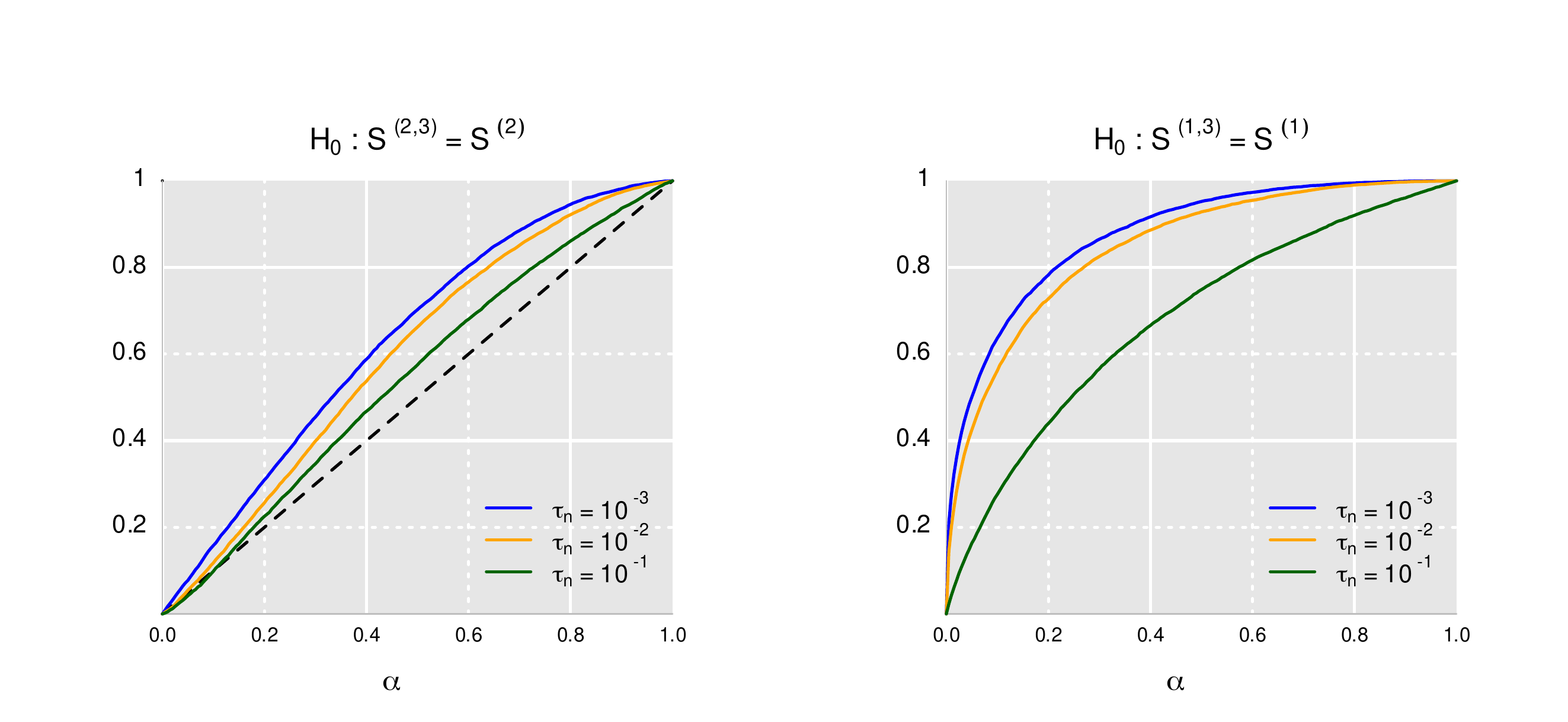}
\caption{\small Estimated significance level for $H_0: S^{(2,3)} = S^{(2)}$ (left) and power for $H_0: S^{(1,3)} = S^{(1)}$ (right) for a TSVD threshold obtained with $\tau_n = 10^{-3}, 10^{-2}$ and $10^{-1}$ (see Eq. \eqref{tau}). The sample size is $n = 60$ and design size $K=10$. }
\label{fig:t}
\end{figure}

Finally, the calibration of the regularization threshold used in the estimation of $ \Gamma(\mathbf x)$ has a non negligible impact on the quality of the test. In Figure \ref{fig:t}, we show the difference in both power and significance level for three different values of $\tau_n$. In this case, the rule of thumb gives the somewhat conservative $\tau_n = 0.1 n^{-1/3} \approx 0.026$, which ensures a reliable test in term of significance level. Remark that although both thresholds $\tau_n = 10^{-2}$ and $\tau_n = 10^{-1}$ lead to similar and somewhat accurate levels, we observe a significant improvement in term of power. This suggests that the EP test procedure has room for improvement, at least through optimizing the choice of the regularization threshold.

\subsection{Fuel consumption for aeronautical missions}

Fuel consumption in aeronautics has always been a key issue for the aeronautical and aerospace sectors as one of the main cost for airlines. Much efforts have been made in the past decades to reduce the airplanes fuel burnt, both at the aircraft design stage (by reducing mass, improving aerodynamics or optimizing engines) and during the operations (by searching in the best trajectories - both the ground track and the vertical profile - or by optimizing the quantity of fuel loaded to fly the distance and in the same time, meet the operational safety regulations). \\

The question of quantifying the impact of the operational variability and aircraft design on fuel consumption was raised in \cite{Paper:Peteilh2016, Peteilh:2017aa}. Operational variability can be measured from the disturbance of the moment chosen to climb, among the seven flight levels (FL) available in the cruise altitude ladder, yielding seven input variables $X_\text{fl1}, ..., X_\text{fl7}$. 
From a design point of view, a potential solution to make the airplane more robust to this variability appears to be a local modification of the airplane polar curve, whose distribution in the model depends on a position parameter $X_\text{cz}$ and a shape parameter $X_\text{lod}$. \\

The model used to create the experiment is based on the MARILib tools~\cite{Druot:2019aa, druot2022hydrogen}. 
In this study, a four-engine turbofan long range type of aircraft 
has been chosen and its design frozen except for the additional local aerodynamic parameters $X_\text{cz}$ and $X_\text{lod}$.
One reference mission is calculated with all input parameters set to zero which represent the neutral position and the optimized flight profile with the basic aerodynamics. The other flights calculated have a perturbed flight profile and locally improved aerodynamics. A total of one thousand flights have been calculated.\\

We discuss a step-by-step methodology to assess the importance of each input on the excess fuel consumption. Each step is aimed to describe one possible way to interpret and proceed based on the tests' results.\\

\noindent \textbf{Preliminary remarks. } 

\begin{itemize}
\item A first analysis, conducted on incorrectly generated data, concluded that none of the local aerodynamic inputs $X_\text{cz}, X_\text{lod}$ had any measurable impact on the excess fuel consumption, either directly or via interactions with the other inputs. This observation convinced the expert to revisit the source code where an error was found and corrected. The methodology applied to the correctly generated data is described below. 

\item Both approaches for the computation of the p-values, namely that of Eq.\eqref{eq:stat_test1} and the less time consuming Eq.\eqref{eq:stat_test2} were considered with similar conclusions in all cases. For sake of simplicity, we only present the results of the first p-value since it is our recommended approach when time and computational resources allow it. The quantiles of the generalized $\chi^2$ distribution where approximated by Monte-Carlo with using a sample of size $10000$. 

\item The difficult issue of calibrating the significance level to account for multiple testing has not been considered in this study. A rigorous way to achieve asymptotically exact multiple tests might be achievable from deriving a joint limit distribution of the empirical processes associated to different inputs sets $u$. While this question has not been discussed in this paper, it may be investigated in a future work. 
\end{itemize}

\noindent \textbf{One-dimensional analysis. } The behavior of the excess fuel consumption $Y$ with respect to each of the nine inputs is shown in Figure~ \ref{fig:plots2}.

\begin{figure}[h!]
\centering
\includegraphics[width = 1\textwidth]{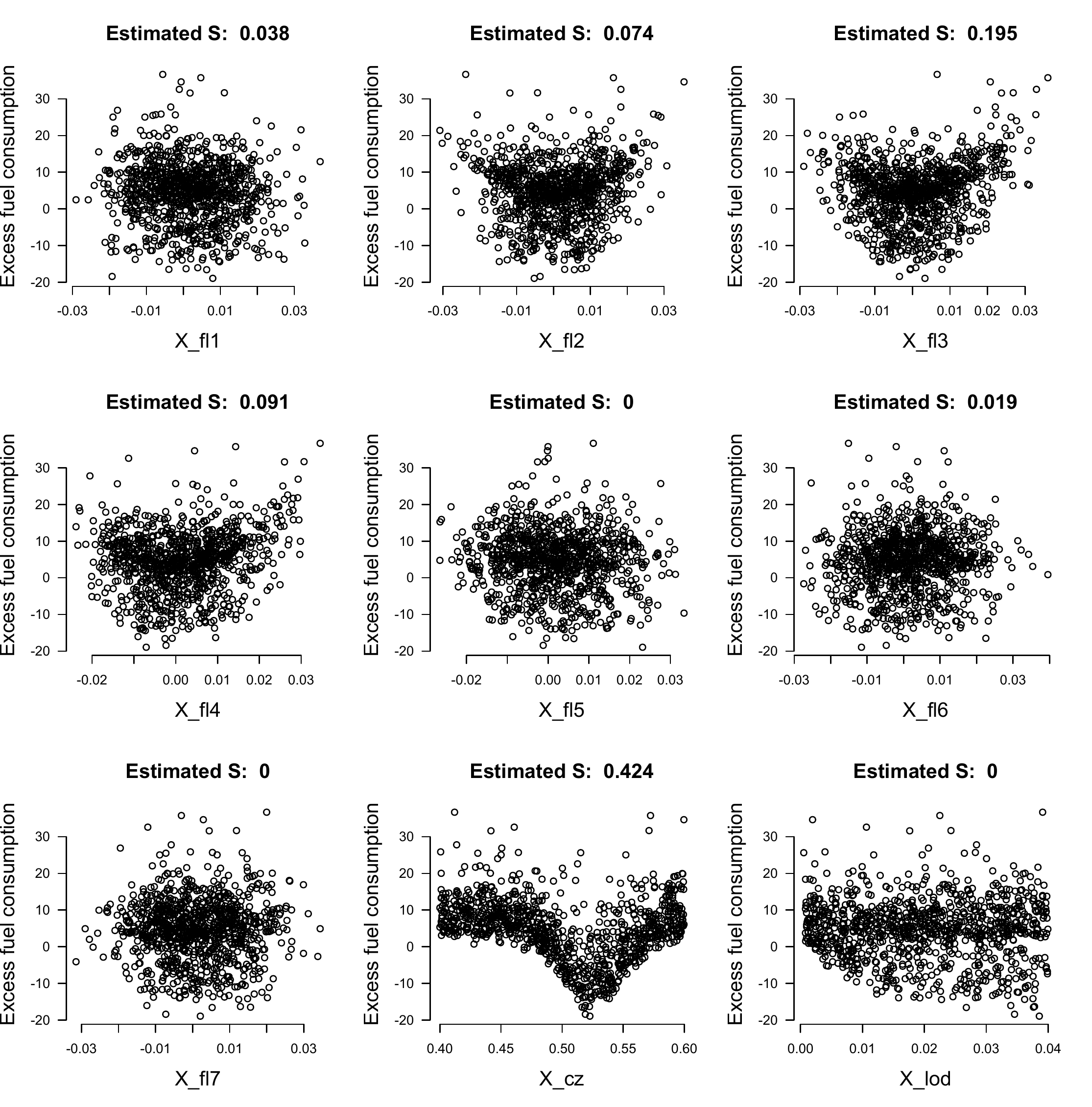}
\caption{\small Bi-variate representation of the excess fuel consumption with respect to each of the inputs $X_{\text{fl1}}, ..., X_{\text{fl7}}, X_{\text{cz}}, X_{\text{lod}}$, with the associated rank-based estimators of the simple Sobol indices from \cite{ggkl21}.}
\label{fig:plots2}
\end{figure}

\begin{table}[H]
\begin{center}
{\renewcommand{\arraystretch}{1.5}
\begin{tabular}{|c|ccccccccc|}
\hline 

 & $X_{\text{fl1}} $  & $X_{\text{fl2}} $  & $X_{\text{fl3}} $  & $X_{\text{fl4}} $ &  $X_{\text{fl5}}$ &  $X_{\text{fl6}}$ &  $X_{\text{fl7}}$ & $X_{\text{cz}}$ & $X_{\text{lod}}$ \\
 \hline 
 p-value & 0.009 & 0 & 0 & 0 & 0.029 & 0.285 & 0.525 & 0 & 0.105 \\
 \hline
\end{tabular}
}
\caption{\small p-values of the tests for the simple hypothesis $H_0: S^{X} = 0$ calculated for the nine inputs $X = X_{\text{fl1}}, ..., X_{\text{fl7}}, X_{\text{cz}}, X_{\text{lod}}$ individually. 
}
\end{center}
\end{table}

Here, the inputs $X_{\text{fl2}}, X_{\text{fl3}}, X_{\text{fl4}}$ and $X_{\text{cz}}$ are highly significant individually while $X_{\text{fl6}}, X_{\text{fl7}}$ and $X_{\text{lod}}$ do not seem to have an impact. The conclusions for $X_{\text{fl1}} $ and $X_{\text{fl5}}$ are more ambiguous. \\

\noindent \textbf{Model validation. } We choose to conserve only the four highly significant inputs and question the validity of the non-parametric model
$$ \mathbb E( Y | X_{\text{fl1}}, ..., X_{\text{fl7}}, X_{\text{cz}}, X_{\text{lod}})  = f(X_{\text{fl2}}, X_{\text{fl3}}, X_{\text{fl4}}, X_{\text{cz}}).  $$
In this context, the test can be used as a tool for non-parametric variable selection where the influence of each input, either added to or removed from the model, can be tested individually. The results show that the four included inputs are all highly significant, while the other inputs are summarized in the following table.

\begin{table}[H]
\begin{center}
{\renewcommand{\arraystretch}{1.5}
\begin{tabular}{|c|ccccc|}
\hline 

 & $X_{\text{fl1}} $  & $X_{\text{fl5}}$ &  $X_{\text{fl6}}$ &  $X_{\text{fl7}}$ &  $X_{\text{lod}}$ \\
 \hline 
 p-value  & 0.026 & 0.249 & 0.924 & 0.252 & 0.004 \\
 \hline
\end{tabular}
}
\caption{\small p-values of the EP test to assess the significance of the inputs in the non-parametric model with $X_{\text{fl2}}, X_{\text{fl3}}, X_{\text{fl4}}$ and $X_{\text{cz}}$. }
\end{center}
\end{table}

Based on these results, one can argue that the question of including the first output $X_{\text{fl1}} $ to the model remains open. More importantly, the last input $X_{\text{lod}}$ has become quite significant, a phenomenon that had not been observed in the previous steps of the analysis. Including this input to the non-parametric model yields to following results, rather stable compared to the previous model.

\begin{table}[H]
\begin{center}
{\renewcommand{\arraystretch}{1.5}
\begin{tabular}{|c|cccc|}
\hline 

 & $X_{\text{fl1}} $  & $X_{\text{fl5}}$ &  $X_{\text{fl6}}$ &  $X_{\text{fl7}}$  \\
 \hline 
 p-value & 0.025 & 0.099 & 0.924 & 0.514  \\
 \hline
\end{tabular}
}
\caption{\small p-values of the EP test to assess the significance of the (non-included) inputs in the non-parametric model with $X_{\text{fl2}}, X_{\text{fl3}}, X_{\text{fl4}}, X_{\text{cz}}$ and $X_{\text{lod}}$. }
\end{center}
\end{table}

\noindent \textbf{Global model significance. } Denoting by $S$ the global Sobol index of $Y$ with respect to the nine inputs, the validity of this final model can be assessed by testing the hypothesis
$$ H_0: S^{\text{ fl2, fl3, fl4, cz, lod}} = S . $$
We obtained p-values of approximately $0.276$ and $0.645$ for the two versions of the test, suggesting that the five inputs $X_{\text{fl2}}, X_{\text{fl3}}, X_{\text{fl4}}, X_{\text{cz}}, X_{\text{lod}}$ are in fact sufficient to explain the excess fuel consumption. \\

\noindent \textbf{Conclusion. } The test has been used on aeronautical data generated from a meta-model for planes fuel consumption. A step-by-step methodology for non-parametric model selection was successful in solving several problems:
\begin{itemize}
\item A previous analysis which concluded to the absence of measurable impact of aerodynamic inputs shed light on an error in the source computer code. This issue was solved and new data were generated from the corrected meta-model.
\item The individual impact of each input can be assessed to provide a preliminary idea of the relevant inputs.   
\item The significance of each input can be assessed in any given model, whether the input is included to the model or not. Then, the decision to include or remove an input can be made based on the results of the tests. Iterating the process leads to a step-wise variable selection process in a non-parametric setting. 
\item A global significance test can be performed to validate a final model. 
\end{itemize}

This example shows one possible approach to use the test procedure for non-parametric variable selection although the various steps of the process are left to the practitioner's interpretation. The method resulted in a selection of only five inputs among nine that were sufficient to explain the whole influence on the excess fuel consumption. In particular, the local aerodynamic parameters $X_{\text{cz}}, X_{\text{lod}}$) have been confirmed as potential relevant levers to get to a more operationally robust airplane. 

\bibliographystyle{plain}
\bibliography{refs}

\begin{thebibliography}{10}

\bibitem{betancourt:hal-01998724}
J.~Betancourt, F.~Bachoc, T.~Klein, D.~Idier, R.~Pedreros, and J.~Rohmer.
\newblock {Gaussian process metamodeling of functional-input code for coastal
  flood hazard assessment}.
\newblock {\em {Reliability Engineering and System Safety}}, 198, June 2020.

\bibitem{rocquigny2008uncertainty}
E.~De~Rocquigny, N.~Devictor, and S.~Tarantola.
\newblock {\em Uncertainty in industrial practice}.
\newblock Wiley Online Library, 2008.

\bibitem{Druot:2019aa}
Thierry~Y Druot, Mathieu Belleville, Pascal Roches, Fran{\c c}ois Gallard,
  Nicolas Peteilh, and Anne Gazaix.
\newblock A multidisciplinary airplane research integrated library with
  applications to partial turboelectric propulsion.
\newblock In {\em AIAA Aviation 2019 Forum}, page 3243, 2019.

\bibitem{druot2022hydrogen}
Thierry~Y Druot, Nicolas Peteilh, Pascal Roches, and Nicolas Monrolin.
\newblock Hydrogen powered airplanes, an exploration of possible architectures
  leveraging boundary layer ingestion and hybridization.
\newblock In {\em AIAA Scitech 2022 Forum}, page 1025, 2022.

\bibitem{engl1996regularization}
H.~W. Engl, M.~Hanke, and A.~Neubauer.
\newblock {\em Regularization of inverse problems}, volume 375.
\newblock Springer Science \& Business Media, 1996.

\bibitem{FKL21}
J.C. Fort, T.~Klein, and A.~Lagnoux.
\newblock Global sensitivity analysis and wasserstein spaces.
\newblock {\em SIAM/ASA Journal on Uncertainty Quantification}, 9(2):880--921,
  2021.

\bibitem{ggkl21}
F.~Gamboa, P.~Gremaud, T.~Klein, and A.~Lagnoux.
\newblock Global sensitivity analysis: a new generation of mighty estimators
  based on rank statistics.
\newblock {\em Forthcoming paper in Bernoulli}, 2022.

\bibitem{pickfreeze}
F.~Gamboa, A.~Janon, T.~Klein, A.~Lagnoux, and C.~Prieur.
\newblock Statistical inference for {S}obol {P}ick-{F}reeze {M}onte {C}arlo
  method.
\newblock {\em Statistics}, 50(4):881--902, 2016.

\bibitem{gamboa2016statistical}
F.~Gamboa, A.~Janon, T.~Klein, A~Lagnoux, and C.~Prieur.
\newblock Statistical inference for sobol pick-freeze monte carlo method.
\newblock {\em Statistics}, 50(4):881--902, 2016.

\bibitem{GODA201763}
T.~Goda.
\newblock Computing the variance of a conditional expectation via non-nested
  {M}onte {C}arlo.
\newblock {\em Operations Research Letters}, 45(1):63 -- 67, 2017.

\bibitem{Hoeffding48}
W.~Hoeffding.
\newblock A class of statistics with asymptotically normal distribution.
\newblock {\em Ann. Math. Statistics}, 19:293--325, 1948.

\bibitem{idier:hal-02458084}
D.~Idier, A.l Aurouet, F.~Bachoc, A.~Baills, J.~Betancourt, J.~Durand,
  R.~Mouche, J.~Rohmer, F.~Gamboa, T.~Klein, J.~Lambert, G.~Le~Cozannet,
  S.~Leroy, J.~Louisor, R.~Pedreros, and A.L. V{\'e}ron.
\newblock {Toward a User-Based, Robust and Fast Running Method for Coastal
  Flooding Forecast, Early Warning, and Risk Prevention}.
\newblock {\em {Journal of Coastal Research, Special Issue}}, 95:11--15, 2020.

\bibitem{janon2012asymptotic}
A.~Janon, T.~Klein, A.~Lagnoux, M.~Nodet, and C.~Prieur.
\newblock Asymptotic normality and efficiency of two {S}obol index estimators.
\newblock {\em ESAIM: Probability and Statistics}, 18:342--364, 1 2014.

\bibitem{Kucherenko2017different}
S.~Kucherenko and S.~Song.
\newblock Different numerical estimators for main effect global sensitivity
  indices.
\newblock {\em Reliability Engineering \& System Safety}, 165:222--238, 2017.

\bibitem{mazo2021trade}
Gildas Mazo.
\newblock A trade-off between explorations and repetitions for estimators of
  two global sensitivity indices in stochastic models induced by probability
  measures.
\newblock {\em SIAM/ASA Journal on Uncertainty Quantification},
  9(4):1673--1713, 2021.

\bibitem{Owen13}
Art~B. Owen.
\newblock Better estimation of small sobol' sensitivity indices.
\newblock {\em ACM Trans. Model. Comput. Simul.}, 23(2):11:1--11:17, May 2013.

\bibitem{pearson1915partial}
K.~Pearson.
\newblock On the partial correlation ratio.
\newblock {\em Proceedings of the Royal Society of London. Series A, Containing
  Papers of a Mathematical and Physical Character}, 91(632):492--498, 1915.

\bibitem{Paper:Peteilh2016}
Nicolas Peteilh.
\newblock Towards a robust multidisciplinary design optimization model for the
  airplane in the air transport system.
\newblock In {\em AEGATS `16, Advanced Aircraft Efficiency in a Global Air
  Transport System}, April 2016.

\bibitem{peteilh:hal-02866381}
Nicolas Peteilh, Thierry Klein, Thierry~Y Druot, Nathalie Bartoli, and Rhea~P
  Liem.
\newblock Challenging top level aircraft requirements based on operations
  analysis and data-driven models, application to takeoff performance design
  requirements.
\newblock In {\em AIAA Aviation 2020 Forum}, page 3171, 2020.

\bibitem{Peteilh:2017aa}
Nicolas Peteilh, Marcel Mongeau, Christian Bes, Thierry Druot, and M{\'e}lanie
  Conderolle-Lestremau.
\newblock Modeling operational variability for robust multidisciplinay design
  optimization.
\newblock In {\em 18th AIAA/ISSMO Multidisciplinary Analysis and Optimization
  Conference}, page 4328, 2017.

\bibitem{saltelli-sensitivity}
A.~Saltelli, K.~Chan, and E.M. Scott.
\newblock {\em Sensitivity analysis}.
\newblock Wiley Series in Probability and Statistics. John Wiley \& Sons, Ltd.,
  Chichester, 2000.

\bibitem{sobol1993}
I.~M. Sobol.
\newblock Sensitivity estimates for nonlinear mathematical models.
\newblock {\em Math. Modeling Comput. Experiment}, 1(4):407--414 (1995), 1993.

\bibitem{sobol2001global}
I.~M. Sobol.
\newblock {Global sensitivity indices for nonlinear mathematical models and
  their {M}onte {C}arlo estimates}.
\newblock {\em Mathematics and Computers in Simulation}, 55(1-3):271--280,
  2001.

\bibitem{Sudret2008global}
B.~Sudret.
\newblock {Global sensitivity analysis using polynomial chaos expansions}.
\newblock {\em Reliability Engineering \& System Safety}, 93(7):964--979, 2008.

\end{thebibliography}

\end{document}